\documentclass[12pt]{article}

\usepackage[T1]{fontenc}
\usepackage{avant}
\usepackage[cp1250]{inputenc}
\usepackage{amsmath,amssymb,amsthm}

\newtheorem{theorem}{Theorem}[section]
\newtheorem{lemma}[theorem]{Lemma}

\theoremstyle{definition}
\newtheorem{definition}[theorem]{Definition}
\newtheorem{corollary}[theorem]{Corollary}
\newtheorem{proposition}[theorem]{Proposition}

\theoremstyle{remark}

\numberwithin{equation}{section}

\begin{document}

\title{On solvable compact Clifford-Klein forms}
\author{Maciej Boche\'nski, Aleksy Tralle}

\maketitle{}

\begin{abstract}
In this article we prove that under certain assumptions,  a reductive homogeneous space $G/H$ does not admit  a solvable compact Clifford-Klein form. This generalizes the well known non-existence theorem of Benoist for nilpotent Clifford-Klein forms. This generalization  works for a particular class of homogeneous spaces determined by  ``very regular''  embeddings of $H$ into $G$.

\end{abstract}

\tableofcontents

\section{Introduction}
This article is motivated by a circle of problems related to Clifford-Klein forms of homogeneous spaces \cite{b, iw, kas, ko, kob1, kdef, ko1, ko2, lazim, lamo, ow, w, zim}. Assume that we are given a homogeneous space $G/H$ of a (reductive) real Lie group $G$ and a closed subgroup $H\subset G$. We say that $G/H$ admits a compact Clifford-Klein form, if there exists a discrete subgroup $\Gamma\subset G$ which acts properly and co-compactly on $G/H$ by left translations (that is, the double coset $\Gamma\setminus G/H$ is compact). Throughout this paper, we will use also the terminology of \cite{lazim} saying that $\Gamma$ is a co-compact lattice in $G/H$, by analogy with lattices in Lie groups.  

If $G$ is reductive and $H$ is compact then the space $G/H$ always admits a compact Clifford-Klein form. One can simply choose a co-compact lattice $\Gamma \subset G.$ On the other hand if $H$ is a non-compact subgroup, the space $G/H$ does not necessarily admit compact Clifford-Klein forms.

One of the important and challenging problems in the whole area is  Kobayashi's conjecture, which we now describe (see Conjecture 3.3.10 in \cite{ko}). Assume that $G/H$ is a homogeneous space of reductive type (see Definition \ref{def:red-type}). We say that $G/H$ admits a standard Clifford-Klein form, if there exists a reductive Lie subgroup $L\subset G$ such that $L$ acts properly on $G/H$ and $L\setminus G/H$ is compact. The Kobayashi's conjecture states that for semisimple real Lie groups $G$ and homogeneous spaces $G/H$ of reductive type, the existence of a compact Clifford-Klein form on $G/H$ implies the existence of a standard one. Note that the conjecture does not say that all compact Clifford-Klein forms are standard. There are examples of non-standard ones (see \cite{kas}, \cite{kdef}), obtained by a deformation of the latter. 
\noindent
In fact all known examples of standard Clifford-Klein forms  can be obtained in the following way. Assume that there exists a reductive Lie subgroup $L\subset G$ such that $G=L\cdot H$,  and $L\cap H$ is compact. Under these assumptions we see that 
\begin{itemize}
\item $L$ acts transitively on $G/H$, and, therefore there is a diffeomorphism $G/H\simeq L/(L\cap H)$
\item since $L\cap H$ is compact, any co-compact lattice $\Gamma\subset L$ acts properly and co-compactly on $L/(L\cap H)$, and, hence, on $G/H$.
\end{itemize}
Notice that the Kobayashi conjecture indicates that compact Clifford-Klein forms of non-compact semisimple homogeneous spaces $G/H$ of reductive type are rare and of a special nature. Our motivation  is to eliminate the possibility of obtaining Clifford-Klein forms from double quotients of connected subgroups that are solvable, and, therefore, to obtain further evidence for the conjecture. In this context, an important starting point is a theorem of Benoist \cite{b}, which shows that {\it nilpotent groups} cannot yield Clifford-Klein forms. In greater detail, the following holds. 
\begin{theorem}[Benoist]\label{thm:benoist} Assume that $G/H$ is a non-compact homogeneous space of a semisimple real Lie group. If a nilpotent subgroup ${N}\subset G$ acts properly on $G/H$, then $N\setminus G/H$ cannot be compact.
\end{theorem}

\noindent
Our main result is the following.
\begin{theorem} 
Let $G/H$ be a homogeneous space such that $G$ is a connected semisimple linear Lie group of non-compact type and $H$ is the semisimple part of the Levi factor of some parabolic subgroup of $G$. If $\Gamma\setminus G/H$ is a compact Clifford-Klein form, then $\Gamma$ cannot be (virtually) solvable.

\label{thm:Gamma}
\end{theorem} 
Throughout this work we denote by $SL(n,\mathbb{R})/SL(m,\mathbb{R})$ the homogeneous space determined by the standard embedding of $SL(m,\mathbb{R})$ into $SL(n,\mathbb{R})$ of the form $A\rightarrow\operatorname{diag}(A,E_{n-m})$, where $E_{n-m}$ is the identity matrix. In a similar way, one can describe the embedding of the exceptional Lie group $E_7^{V}$ into $E_8^{VIII}$.
\begin{corollary}
The following homogeneous spaces do not admit virtually solvable compact Clifford-Klein forms
$$SL(n,\mathbb{R})/SL(m,\mathbb{R}), \ E_{8}^{VIII}/E_{7}^{V}, \ n>m>1.$$
\end{corollary}

Notice that $SL(n,\mathbb{R})/SL(m,\mathbb{R})$ is an important ``test example'' in the theory of compact Clifford-Klein forms. The first results concerning co-compact lattices of these spaces were due to Kobayashi \cite{kobr1, kobr2, kobr3}. Benoist proved in \cite{b} that the homogeneous space $SL(2n+1,\mathbb{R})/SL(2n,\mathbb{R}),$ $n>1$ does not admit compact Clifford-Klein forms at all. The general case $SL(n,\mathbb{R})/SL(m,\mathbb{R})$ is still open and of significant interest (see \cite{b}, \cite{lazim}, \cite{lamo}, \cite{sha} for partial results).

\section{Preliminaries}
Throughout this article we use the basics of Lie theory without further explanations closely following \cite{ov}. Also we denote the Lie algebras of Lie groups $G$,$H$,... by the corresponding Gothic letters $\mathfrak{g},\mathfrak{h}$,... etc. As usual, the symbol $\mathfrak{g}^c$ denotes the complexification of a real Lie algebra $\mathfrak{g}$, the centralizer of a Lie subalgebra $\mathfrak{a}$ in $\mathfrak{g}$ is denoted by $\mathfrak{z}_{\mathfrak{g}}(\mathfrak{a})$. We also use relations between real Lie groups and algebraic groups following \cite{mar, ov2, pr, w-book, vgs}. 

\noindent
Let $G$ be a real connected linear semisimple Lie group with the Lie algebra $\mathfrak{g}$ and $H\subset G$ a closed connected subgroup with the Lie algebra $\mathfrak{h}$. The Lie algebra $\mathfrak{g}$ has a Cartan decomposition (and the corresponding Cartan involution)
$$\mathfrak{g}=\mathfrak{k}+\mathfrak{p}$$
where $\mathfrak{k}$ is a maximal compact Lie subalgebra. Choose a maximal abelian subspace $\mathfrak{a}\subset\mathfrak{p}$. Note that all such subalgebras are conjugate in $\mathfrak{g}$. There is a maximal abelian subalgebra $\mathfrak{t}\subset\mathfrak{g}$ (called the {\it split Cartan subalgebra}) of the form

 $$\mathfrak{t}=\mathfrak{t}_{\mathfrak{k}}+\mathfrak{a},$$
where $\mathfrak{t}_{\mathfrak{k}}$ denotes a maximal abelian subalgebra of $\mathfrak{z}_{\mathfrak{k}}(\mathfrak{a})$. Then $\mathfrak{t}^c$ is a Cartan subalgebra of the complexification $\mathfrak{g}^c$, and, therefore, determines the root system $\Sigma$ of $\mathfrak{g}^c$. The (non-zero) restrictions of $\alpha\in\Sigma$ on $\mathfrak{a}$ form a system of restricted roots $\Delta$ (which is an ``abstract'' root system itself). In this article we will use only restricted roots. Therefore, throughout this paper we will call them ``real roots''. Recall that the real root decomposition is given by the formula
 
  $$\mathfrak{g}=\mathfrak{t}_{\mathfrak{k}}+\mathfrak{a}+\sum_{\alpha\in\Delta}\mathfrak{g}_{\alpha},$$
  where $\mathfrak{g}_{\alpha}$ are the weight subspaces of the adjoint representation of $\mathfrak{a}$ ($\mathfrak{g}_{\alpha}$ need not to be one-dimensional). 
 Note that we will use the basics of the theory of root systems in the real case. 

The Weyl group $W$ of $\mathfrak{g}$ is the finite group of orthogonal transformations of $\mathfrak{a}$ (with respect to the Killing form of $\mathfrak{g}$) generated by reflections in hyperplanes $C_{\alpha}:=\{ X\in \mathfrak{a} \ | \ \alpha (X)=0  \}$ for $\alpha \in \Delta.$ The following is well known.

\begin{proposition}[\cite{ov}, Proposition 4.2, Chapter 4]\label{prop:weyl}
The group $W$ coincides with the group of transformations induced by automorphisms $\operatorname{Ad}(k)$ $(k\in N_{K}(\mathfrak{a}))$ and also with the group of transformations induced by automorphisms $\operatorname{Ad}(g)$ $(g\in N_{G}(\mathfrak{a})).$ Therefore
$$W\cong N_{K}(\mathfrak{a})/Z_{K}(\mathfrak{a})\cong N_{G}(\mathfrak{a})/Z_{K}(\mathfrak{a}).$$
\end{proposition}

\noindent
The Weyl chamber determined by a set of all positive roots with respect to some fixed $\mathfrak{a}$ is denoted by $\mathfrak{a}^+$.  

\begin{definition}\label{def:red-type}
Let $G$ be a connected real semisimple linear Lie group and $H$ be a closed subgroup with finitely many connected components. The homogeneous space $G/H$ is of {\it reductive type}, if there exists a Cartan involution $\theta$ of $\mathfrak{g}$ such that $\theta (\mathfrak{h})=\mathfrak{h}.$
\end{definition}
\noindent Notice  that if $G/H$ is of reductive type then $H$ is a reductive Lie group. 
Also, we use a relation between Lie groups and linear algebraic $\mathbb{R}$-groups (see \cite{mar}).  If ${\bf {G}}\subset GL(n,\mathbb{C})$ is an algebraic $\mathbb{R}$-group, then $G={\bf{G}_{\mathbb{R}}}={\bf G}\cap GL(n,\mathbb{R})$ is a Lie group with a finite number of connected components.

Let $X$ be a Hausdorff topological space and $\Gamma$ a topological group acting on $X$. We say that an action $\Gamma$ on $X$ is {\it proper}, if for any compact subset $S\subset X$ the set
$$\{\gamma\in\Gamma\,|\,\gamma (S)\cap S\not=\emptyset\}$$
is compact. In particular, in this article we consider the proper actions of $\Gamma\subset G$ on $X=G/H$ by left translations. 
It easily follows from the definition that if a closed connected Lie subgroup $L\subset G$ acts properly on $G/H$ then any orbit of $L$ is closed.

We will also need the following results on proper actions. For a connected Lie group $J$ define $d(J):=\dim\,J-\dim\,K_{J},$ where $K_{J}$ is a maximal compact subgroup of $J.$

\begin{theorem}[\cite{ko2}, Theorem 4.7 and \cite{ow}, Theorem 3.4]
Let $A\subset G$ and $B \subset G$ be closed connected subgroups of $G.$ If $A$ acts properly on $G/B$, then
$$A\setminus G/B \,\,\textrm{is compact if and only if}\,\, d(G)=d(A)+d(B).$$

\label{kk1}
\end{theorem}

Let $\mathfrak{g}=\mathfrak{k}\oplus\mathfrak{p}$ be a Cartan decomposition, and let $\Delta^+$ be a subset of positive roots of the real root system $\Delta$ with respect to a fixed $\mathfrak{a}\subset\mathfrak{p}$. Set

$$\mathfrak{n}=\sum_{\alpha\in\Delta^+}\mathfrak{g}_{\alpha}.$$
One can easily see that $\mathfrak{n}$ is a real nilpotent subalgebra of $\mathfrak{g}$. Also, $\mathfrak{a}+\mathfrak{n}$ is a solvable subalgebra of $\mathfrak{g}$. One obtains the Iwasawa decomposition
$$\mathfrak{g}=\mathfrak{k}+\mathfrak{a}+\mathfrak{n},$$
where $\mathfrak{k}$ is a maximal compact subalgebra of $\mathfrak{g}$. 
It follows from Definition \ref{def:red-type}, that if $G/H$ is of reductive type, then $\mathfrak{h}$ admits a Cartan decomposition compatible with that of $\mathfrak{g}$:
$$\mathfrak{h}=\mathfrak{k}_h+\mathfrak{p}_h,\ \mathfrak{k}_h=\mathfrak{k}\cap\mathfrak{h}, \ \mathfrak{p}_h=\mathfrak{h}\cap\mathfrak{p}.$$
In the same way, $\mathfrak{h}$ admits a {\it compatible Iwasawa decomposition}
$$\mathfrak{h}=\mathfrak{k}_h+\mathfrak{a}_h+\mathfrak{n}_h, \ \mathfrak{n}_h=\mathfrak{n}\cap\mathfrak{h}, \ \mathfrak{a}_h=\mathfrak{a}\cap\mathfrak{h}.$$ 
 
Now let $G$ be a connected semisimple Lie group whose Lie algebra is $\mathfrak{g}$. There exists a connected compact Lie subgroup $K\subset G$ and simply connected Lie subgroups $A$ and $N$ whose Lie algebras are $\mathfrak{a}$ and $\mathfrak{n}$ such that
$$G=K\cdot A\cdot N$$
is a topological decomposition into a direct product of subgroups. This decomposition is the (global) Iwasawa decomposition of $G$. In the same manner, we obtain the {\it compatible} global Iwasawa decomposition of $H$: $H=K_h\cdot A_h\cdot N_h$ (the meaning of the symbols is clear). We need  also one more decomposition, the Cartan decomposition, $G=K\cdot A\cdot K$. 

Consider $G$ as a semisimple group of $\mathbb{R}$-points of an algebraic $\mathbb{R}$-group ${\bf G}$, and $A$ as a subgroup of $\mathbb{R}$-points of a maximal algebraic torus ${\bf A}$. Then one can use the root system $\Delta=\Delta({\bf A},{\bf G})$ of ${\bf G}$ with respect to ${\bf A}$ and define the global Weyl chamber

 $$A^+=\{a\in A\,|\,\chi(a)>0,\,\text{for any}\,\chi\in\Delta^+\}.$$
 For more details we refer to \cite{b}. Note that we use the same symbol to denote the root systems for ${\bf G}$ and $\mathfrak{g}$.  This easily yields the decomposition
 $$G=K\cdot A^+\cdot K$$
 (which is also called the Cartan decomposition). Thus, for any element $g\in G$ there is an element $a_g\in A^+$ such that $g\in K\cdot a_g\cdot K$. This element is unique, hence there is a well defined map $\mu: G\rightarrow A^+$ given by the formula
 $$\mu(g)=a_g,$$
called the Cartan projection. The function $\mu$ is continuous and proper (that is, the preimage of a compact set is compact). Also, using the diffeomorphism $\log: A^+\rightarrow \mathfrak{a}^+$ one obtains a proper map $\mu: G\rightarrow \mathfrak{a}^+$, where $\mathfrak{a}^+\subset\mathfrak{a}$ is a closed Weyl chamber in $\mathfrak{a}$.  Note that we will denote both maps by the same letter, because we will use it only as a map with $\mathfrak{a}^+$ as a target.

We will need the following characterization of a properness of an action of a subgroup on a homogeneous space.
\begin{theorem}[\cite{b}, Corollary 5.2 and \cite{kob1}, Theorem 1.1]\label{thm:c-proj}
Let $A,B \subset G$ be closed connected subgroups of $G$ and $\mu$ the Cartan projection in $G.$ The subgroup $A$ acts properly on $G/B$ if and only if
$$\mu (A) \cap ( \mu (B) + C )$$
is bounded for every compact subset $C \subset \mathfrak{a}.$
\label{kb}
\end{theorem}
In this paper we are following an approach of Toshiyuki Kobayashi \cite{ko2}. The latter is based on the following observation.
\begin{proposition}[\cite{ko2}, Lemma 2.3]\label{prop:discrete-connected} Let a real Lie group $G$ act on a locally compact Hausdorff space $X$ and $\Gamma$ be a uniform lattice in $G$. Then
\begin{enumerate}
\item The $G$-action on $X$ is proper if and only if  the $\Gamma$-action on $X$ is proper.
\item $G\setminus X$ is compact if and only if $\Gamma\setminus X$ is compact.
\end{enumerate}
\end{proposition}
We will use the following Lemma and Theorem  in the proof of the main result. 
\begin{lemma}[\cite{ko1}, Lemma 1.3]
Let $G_{1},G_{2}$ be locally compact groups and $L_{1},H_{1}\subset G_{1},$ $L_{2},H_{2} \subset G_{2}$ be closed subgroups. Assume that $f:G_{1}\rightarrow G_{2}$ is a continuous homomorphism such that $f(L_{1})\subset L_{2},$ $f(H_{1})\subset H_{2},$ $f(L_{1})$ is closed in $G_{2}$ and $L_{1} \cap Kerf$ is compact. Then if the action of $L_{2}$ on $G_{2}/H_{2}$ is proper then the action of $L_{1}$ on $G_{1}/H_{1}$ is proper.
\label{kol}
\end{lemma}
\begin{theorem}[\cite{w}, Theorem 6.2]\label{thm:orbit-decomp}
Let $M$ and $N$ be connected subgroups of a connected, simply connected, solvable Lie group $S$. If $M\backslash S/N$ is compact, and every orbit of $M$ is closed, then $S=MN$.
\label{own}
\end{theorem}
\noindent
Also, in the proof of the main result we will use the fact that homogeneous spaces $G/H$ of reductive type and of a maximal real rank cannot admit proper actions of finite discrete subgroups, and, hence, cannot admit compact Clifford-Klein forms. This fact is called the Calabi-Markus phenomenon.
\begin{theorem}[\cite{ko2}, Corollary 4.4]\label{thm:calabi-m} Let $G/H$ be a homogeneous space of reductive type. If $\text{rank}_{\mathbb{R}}G=\text{rank}_{\mathbb{R}}H$, then only finite groups can act properly on $G/H$. In particular, such $G/H$ cannot have compact Clifford-Klein forms.
\end{theorem}
In the proof of Theorem \ref{thm:Gamma} we will need the Jacobson-Morozov theorem (see \cite{cm}, Theorem 9.2.1). We say that a triple $(H,X,Y)$ of vectors in $\mathfrak{g}$ is an $\mathfrak{s}\mathfrak{l}_2$-{\it triple}, if
$$[H,X]=2X,\,[H,Y]=-2Y,\,[X,Y]=H.$$
\begin{theorem}[Jacobson-Morozov]\label{thm:j-m} Let $\mathfrak{g}$ be a real semisimple Lie algebra and $X$ be a non-zero nilpotent element. Then there exists an $\mathfrak{s}\mathfrak{l}_2$-triple $(H,X,Y)$ in $\mathfrak{g}$.
\end{theorem}

Since we consider real Lie algebras, we use the following definition of a {\it parabolic subgroup} $Q\subset G$: it is parabolic, if its Lie subalgebra $\mathfrak{q}$ is parabolic. We say that $\mathfrak{q}\subset\mathfrak{g}$ is parabolic, if the complexification $\mathfrak{q}^c$ is a parabolic Lie subalgebra in the complexification  $\mathfrak{g}^c$ of $\mathfrak{g}$. Note that although we define $\mathfrak{q}$ in terms of Lie algebras over $\mathbb{C}$, it admits a complete description in terms of the real root system of $\mathfrak{g}$. We will present the details in Section \ref{sec:proof1}. Also, we refer to \cite{ov} (Chapter 6, Section 1.5) and \cite{o}.

\section{Proof of Theorem \ref{thm:Gamma}}\label{sec:proof1}

\subsection{Zariski closures and syndetic hulls}
An important notion which we use in the proof of Theorem \ref{thm:Gamma} is the notion of a syndetic hull \cite{w-a}.
\begin{definition}
A {\it syndetic hull} of a subgroup $\Gamma$ of a Lie group $G$ is a subgroup $B$ of $G$ such that $B$ is connected, $B$ contains $\Gamma$ and $\Gamma\setminus B/$ is compact.
\end{definition}
One of the important tools used in the proof of Theorem \ref{thm:Gamma} is the following.
\begin{lemma}\label{lemma:syndetic} Assume that $G/H$ is a homogeneous space of  reductive type. If a discrete solvable subgroup $\Gamma\subset G$ acts properly and co-compactly on $G/H$, then it admits a syndetic hull $B$, which is a connected solvable Lie subgroup admitting a uniform  lattice $\Gamma$. The syndetic hull $B$ acts properly and co-compactly on $G/H$.
\end{lemma}
\begin{proof} The proof of Lemma \ref{lemma:syndetic} follows from the theorem.
\begin{theorem}[\cite{mfj}, Section 1.6]\label{thm:syndetic} Let $V$ be a finite-dimensional real vector space and $\Lambda$ a virtually solvable subgroup of $GL(V)$. Then there exists at least one closed virtually solvable subgroup $S\subset GL(V)$ containing $\Lambda$ such that:
\begin{enumerate}
\item $S$ has finitely many components and each component meets $\Lambda$;
\item (syndeticity) there exists a compact set $K\subset H$ such that $S=K\cdot \Lambda$;
\item $S$ and $\Lambda$ have the same Zariski closure in $GL(V)$.
\end{enumerate}
\end{theorem}
\noindent We continue the proof of Lemma \ref{lemma:syndetic} as follows.
Assume  that we are given a homogeneous space of a reductive type $G/H$, and that $G$ is connected and linear, thus, $G\subset GL(V)$. Assume that $\Gamma$ is a solvable discrete subgroup of $G$ acting properly and co-compactly on $G/H$. Consider the Zariski closure $L=\bar\Gamma$. Apply Theorem \ref{thm:syndetic} to $\Gamma$ (instead of $\Lambda$). We obtain that there exists a subgroup $B\subset GL(V)$ such that $B\supset\Gamma$ and $\bar{\Gamma}=\bar B=L$ (that is, we have $B$ instead of $S$ in Theorem \ref{thm:syndetic}). Since $L$ is the Zariski closure of $\Gamma$, it is also solvable. Therefore, since $\bar B=L$, we obtain a (virtually) solvable subgroup $B$ such that $\Gamma\setminus B$ is compact. Consider the connected component $B_0$. Note that $B$ has only finite number of connected components, hence $B_0$ is a connected Lie subgroup in $G$. Clearly, $B_0$ must be solvable. Since the Lie subgroup $B$ contains a uniform lattice $\Gamma$, so does $B_0$. 
Note that the last claim of the lemma follows from Proposition \ref{prop:discrete-connected}.
\end{proof}

\subsection{Factorizations of Lie groups and Lie algebras}
Following Onishchik \cite{o-book} we introduce the notion of the {\it factorizations of Lie groups and Lie algebras}. 
\begin{definition}\label{def:factor} {\rm We say that a triple $(G,H,L)$ of Lie groups is a factorization, if $H$ and $L$ are Lie subgroups of $G$ and $G=H\cdot L$. In the same way, a triple of Lie algebras $(\mathfrak{g},\mathfrak{h},\mathfrak{l})$ is called a factorization, if $\mathfrak{h}$ and $\mathfrak{l}$ are Lie subalgebras of $\mathfrak{g}$, and $\mathfrak{g}=\mathfrak{h}+\mathfrak{l}$.}
\end{definition}

\noindent
We will need the following straightforward result.
\begin{proposition}[\cite{o-book}, Corollary on p. 88]\label{prop:factorization} If a triple of Lie groups is a factorization, so is the triple of their Lie algebras.
\end{proposition}

\subsection{Inclusion $ {B}\subset {TUC_N(T)} $}

Use the notation from the previous subsection. In this subsection we don't use the fact that $B$ is a syndetic hull of $\Gamma$. Instead of that, we consider the following. Assume that we are given a reductive homogeneous space $G/H$ and a closed connected solvable subgroup $B\subset AN$ acting properly and co-compactly on $G/H$. 
Let $T=A\cap (BN), U=B\cap N$ and denote by $C_N(T)$ the centralizer of $T$ in $N$. 
Consider the connected component of the Zariski closure of $B$ and denote it by $\bar B=L$. Denote the Lie algebras of $B, \bar B=L,\bar U, U,\bar T, T, C_N(\bar T), C_N(T)$ by
$$\mathfrak{b},\mathfrak{l}, \bar{\mathfrak{u}}, \mathfrak{u},\bar{\mathfrak{t}}, \mathfrak{t}, \mathfrak{c},\mathfrak{c}_t,$$
respectively.
Following \cite{ow2}, we say that $B$ is {\it compatible with $A$}, if $B\subset TUC_N(T)$.
\begin{lemma}[\cite{ow2}, Lemma 2.3]\label{lemma:oh-witte} If $B$ is a closed connected subgroup of $AN$, then it is conjugate, via an element of $N$, to a subgroup which is compatible with $A$. 
\end{lemma}

\noindent
Thus, in our considerations we will always assume that
$$B\subset TUC_N(T).$$
Lemma \ref{lemma:oh-witte} implies also the following.
\begin{lemma}\label{lemma:alg-inclusion} The following holds:
 $$\mathfrak{u}+\mathfrak{c}=\bar{\mathfrak{u}},\,\mathfrak{t}\subset\bar{\mathfrak{t}},\,[\mathfrak{t},\mathfrak{u}]\subset \mathfrak{u},\,\bar{\mathfrak{t}}\subset\mathfrak{a},\,\bar{\mathfrak{u}}\subset\mathfrak{n}.$$
\end{lemma}
\begin{proof} The proof of this lemma is contained in the proof of Lemma \ref{lemma:oh-witte} (\cite{ow2}, Lemma 2.3). Therefore, we reproduce it for the convenience of the reader.  Write $\bar B=\bar T\rtimes\bar U$, where $\bar U$ is a subgroup of $N$ and $\bar T$ is conjugate, via element of $N$, to a subgroup of $A$. Here we use the well-known fact (\cite{bor}, Theorem 10.6 (4)) that a real connected solvable algebraic group $L$ is a semidirect product
  $$L=\bar T\ltimes \bar U$$
  of a torus, and a unipotent subgroup $\bar U$. We may assume that $\bar T\subset A$ (taking a conjugate, if necessary).  It is proved in \cite{w2} that $[\bar B,\bar B]\subset B\cap N$, which implies $\operatorname{Ad}_G{\bar T}(\bar{\mathfrak{u}})\subset \mathfrak{u}$. Also, $\bar T\subset A$. Clearly, the subalgebra $\bar{\mathfrak{u}}$ is $\operatorname{Ad}_G(\bar T)$-invariant, and the adjoint action of $\bar T$ on $\bar{\mathfrak{u}}$ is completely reducible. It follows that there is a subspace $\mathfrak{c}\subset\bar{\mathfrak{u}}$ such that
$$\operatorname{Ad}_G(\bar T)(\mathfrak{c})=0,\,\text{and}\,\,\mathfrak{u}+\mathfrak{c}=\bar{\mathfrak{u}}.$$
Therefore, $UC_N({\bar T})=\bar U$, so $\bar B=\bar T UC_N(\bar T)$. Let $\pi: AN\rightarrow A$ be the projection with the kernel $N$, and let $T=\pi(B)$. We get
$$T=\pi(B)\subset\pi(\bar B)=\bar T\Rightarrow C_N(T)\supset C_N(\bar T).$$
For any $b\in B$, there exist $t\in\bar T,u\in U$ and $c\in C_N(\bar T)$ such that $b=tuc$. But $uc\in N$, hence $t=\pi(b)\in T$, and, because $C_N(T)\supset C_N(\bar T)$, we obtain $c\in C_N(T)$. Therefore, $b\in TUC_N(T)$. Finally,  
$[\bar{\mathfrak{t}}, \bar{\mathfrak{u}}]\subset\bar{\mathfrak{u}}$. Now, one can easily see from the above considerations, that also   $[\mathfrak{t},\mathfrak{u}]\subset\mathfrak{u}$, the inclusions $\bar{\mathfrak{t}}\subset\mathfrak{a}$ and $\bar{\mathfrak{u}}\subset\mathfrak{n}$ also have been derived (on the Lie group level): the first one is clear, the second follows from $\bar U=UC_N(\bar T)$. 
The proof of both Lemma \ref{lemma:oh-witte} and Lemma \ref{lemma:alg-inclusion} is complete.  
\end{proof}

\subsection{Non-unimodularity of $B$}

\begin{theorem}\label{thm:non-unimodular-b} Let $B\subset AN$ be a compatible subgroup acting properly and co-compactly on a homogeneous space $G/H$ of reductive type such that $H$ is the semisimple part of the Levi factor of some parabolic subgroup of $G.$ Assume that $\text{rank}_{\mathbb{R}}G>\text{rank}_{\mathbb{R}}H$. Then $B$ cannot be unimodular.
\end{theorem}
\begin{proof} 	
 Consider the inclusion $f: AN\hookrightarrow G$ and put $L_1=A_hN_h,\,L_2=AN, H_1=A_hN_h, H_2=H$. Apply Lemma \ref{kol}. This shows that $B$ acts properly on $AN/A_hN_h$. Also, $B$ acts  co-compactly on $AN/A_hN_h$. The latter easily follows from Theorem \ref{kk1}. Indeed, $d(B)+d(H)=d(G)$, since $B$ acts co-compactly on $G/H$. But $d(H)=d(A_hN_h)$, and $d(AN)=d(G)$. Hence, $d(B)+d(A_hN_h)=d(AN)$, and Theorem \ref{kk1} applies.
Therefore, every orbit of this action is closed. Applying Theorem \ref{thm:orbit-decomp} we obtain a decomposition
$$AN=B(A_hN_h)=LA_hN_h. \eqno (1)$$
Applying Proposition \ref{prop:factorization} and the equality $\mathfrak{l}=\bar{\mathfrak{t}}+\bar{\mathfrak{u}}$ we get
$$\mathfrak{a}+\mathfrak{n}=(\bar{\mathfrak{t}}+\bar{\mathfrak{u}})+(\mathfrak{a}_h+\mathfrak{n}_h).$$
Since $\bar{\mathfrak{t}}\subset\mathfrak{a}$ and $\bar{\mathfrak{u}}\subset\mathfrak{n}$ (by Lemma \ref{lemma:alg-inclusion}), one obtains
$$\mathfrak{n}=\mathfrak{n}_h+\bar{\mathfrak{u}}.\eqno (2)$$
 \noindent Clearly, Lemma \ref{lemma:oh-witte} implies 
 
$$\mathfrak{b}\subset \mathfrak{u}+\mathfrak{t}+\mathfrak{c}_{t}. \eqno (3)$$
 Therefore, $(1)$ and $(3)$ imply
$$(\mathfrak{a}+\mathfrak{n})=\mathfrak{b}+(\mathfrak{a}_h+\mathfrak{n}_h)=(\mathfrak{t}+\mathfrak{u}+\mathfrak{c}_{t})+(\mathfrak{a}_{h}+\mathfrak{n}_{h})$$
and since $\mathfrak{c}_{t}\subset \mathfrak{n}$ we have
$$\mathfrak{t}+\mathfrak{a}_{h}=\mathfrak{a}.\eqno(4)$$
Consider the decomposition
  $$(\mathfrak{a}+\mathfrak{n})=\mathfrak{a}+\sum_{\alpha\in\Delta^+}\mathfrak{g}_{\alpha}.$$ 
	 Since $\mathfrak{h}$ is a semisimple part in some parabolic subalgebra of $\mathfrak{g}$ it follows that $\mathfrak{h}$ admits a compatible Iwasawa decomposition (see \cite{ov}, Section 1.5, Chapter 6)  $\mathfrak{h}=\mathfrak{k}_{h}+\mathfrak{a}_{h}+\mathfrak{n}_{h}$ 
	 and
$$(\mathfrak{a}_{h}+\mathfrak{n}_{h})=\mathfrak{a}_{h}+\sum_{\alpha\in\Delta^{+}_{h}}\mathfrak{g}_{\alpha}.$$
	for a subset of positive roots $\Delta^{+}_{h}$ of $\Delta_{h}\subset \Delta,$ where $\Delta_{h}$ is the root system of $\mathfrak{h}$ and $\Delta$ is the root system of $\mathfrak{g}.$ Moreover, there exists $x_{h} \in \mathfrak{a}$ so that
	$$\Delta_{h}=\{ \alpha \in \Delta \ | \ \alpha (x_{h})=0 \}.$$
	Define $\Delta_{x}^{+} \subset \Delta$ by $\Delta_{x}^{+}:= \{ \alpha \in \Delta \ | \ \alpha (x_{h}) >0  \}$ and put
	$$\mathfrak{n}_{x}:= \sum_{\alpha\in\Delta^{+}_{x}}\mathfrak{g}_{\alpha}. $$
	Therefore (since $\Delta^+_x\cup\Delta^+_h=\Delta^+$) we have a decomposition 
	$$\mathfrak{n}=\mathfrak{n}_{h}+\mathfrak{n}_{x}.\eqno (5)$$
	Since $\mathfrak{n}_{h}$ is given by some subset of root spaces of $\mathfrak{g}$, it follows that 
	$$[\mathfrak{t},\mathfrak{n}_{h}] \subset \mathfrak{n}_{h}.$$
	Thus for any $Y\in\mathfrak{t}$ the decomposition $(5)$ is $\operatorname{ad}_Y$-invariant.
	Now we will get one more $\operatorname{ad}_Y$-invariant decomposition. By $(2)$, and the decomposition  $\bar{\mathfrak{u}}=\mathfrak{u}+\mathfrak{c}$ combined with the obvious inclusion $\mathfrak{c}\subset\mathfrak{c}_t$ we obtain
	$$\mathfrak{n}=\mathfrak{n}_h+\mathfrak{u}+\mathfrak{c}_t. \eqno (6)$$
	The decomposition $(6)$ is $\operatorname{ad}_Y$-invariant, although not direct. We understand $\operatorname{ad}_Y$-invariance in the sense that subspaces $\mathfrak{n}_h$ and $\mathfrak{u}+\mathfrak{c}_t$ are $\operatorname{ad}_Y$-invariant, which  follows from $[\mathfrak{t},\mathfrak{u}]\subset\mathfrak{u}$ (Lemma \ref{lemma:alg-inclusion}), the assumptions on $\mathfrak{h}$ and $[\mathfrak{t},\mathfrak{c}_t]=0$. Note that $\mathfrak{u}\cap\mathfrak{n}_h=\{0\}$, since $U\subset B$ and $B$ acts properly on $G/H$. In greater detail, we argue as follows.
\noindent
If $\mathfrak{n}_{h} \cap \mathfrak{u}$ is not trivial,  it contains $\mathbb{R}X$ for some nonzero $X\in \mathfrak{n}.$ It follows from the Jacobson-Morozov Theorem (Theorem \ref{thm:j-m}) that there exists an $\mathfrak{s}\mathfrak{l}_2$-triple that contains $X.$ Since the Cartan projection of $\tilde{N}$ equals the Cartan projection of $\tilde{A}$ for any semisimple, connected Lie group $S$ with the Iwasawa decomposition $S=\tilde{K}\tilde{A}\tilde{N}$ (see Theorem 5.1 in \cite{kos}), we see that $\mu (H) \cap \mu (B)$ is not bounded, a contradiction (Theorem \ref{kb}).

Let $\mathfrak{m}=(\mathfrak{n}_h+\mathfrak{u})\cap\mathfrak{c}_t$. It is straightforward to see that this subspace is $\operatorname{ad}_Y$-invariant. Writing down an $\operatorname{ad}_Y$-invariant decomposition $\mathfrak{c}_t=\mathfrak{m}\oplus\mathfrak{m}'$ one obtains one more $\operatorname{ad}_Y$-invariant (direct) decomposition
$$\mathfrak{n}=\mathfrak{n}_h\oplus\mathfrak{u}\oplus\mathfrak{m}'.\eqno (7)$$
Finally, $(5)$ and $(7)$ together yield, for any $Y \in \mathfrak{t}$,  two decompositions of $\mathfrak{n}$ into invariant subspaces of the endomorphism $\operatorname{ad}Y:\mathfrak{n} \rightarrow \mathfrak{n},$ namely:
	$$\mathfrak{n}=\mathfrak{n}_{h}+\mathfrak{n}_{x} \ \textrm{and} \ \mathfrak{n}=\mathfrak{n}_{h}+\mathfrak{u}+\mathfrak{m}'. \eqno (8)$$
	We will use $(8)$ as follows. Notice that for any $Y\in\mathfrak{t}$ one has $\operatorname{ad}_Y(\mathfrak{b})\subset\mathfrak{b}$. Indeed, 
	$$\operatorname{ad}_Y(\mathfrak{t}+\mathfrak{u}+\mathfrak{c}_t)\subset\mathfrak{u}\subset\mathfrak{b}.$$
	From this and $(8)$, as well as the fact that the trace does not depend on the basis we obtain 
	$$\operatorname{Tr}((\operatorname{ad}_Y)|_{\mathfrak{b}})=\operatorname{Tr}((\operatorname{ad}_Y)|_{\mathfrak{u}})=\operatorname{Tr}((\operatorname{ad}_Y)|_{\mathfrak{n}_x}). \eqno (9)$$

\noindent
	Recall that we assume that  $\mathfrak{b}$ must be unimodular. We will show that there exists $Z=Y+X\in\mathfrak{b}$ such that:
	\begin{itemize} 
	\item $Y\in\mathfrak{t},\,X\in\mathfrak{c}_t\subset\mathfrak{n}$,
	\item $\operatorname{Tr}((\operatorname{ad}_Y)|_{\mathfrak{n}_x})\not=0$.
	\end{itemize}
	Since $\operatorname{ad}_X$ is nilpotent, the latter yields 
	$$\operatorname{Tr}((\operatorname{ad}_Z)|_{\mathfrak{b}})=\operatorname{Tr}((\operatorname{ad}_Y)|_{\mathfrak{b}})=\operatorname{Tr}((\operatorname{ad}_Y)|_{\mathfrak{n}_x})\not=0.$$

	Therefore, for such $Z\in\mathfrak{b},$ $\operatorname{Tr}(\operatorname{ad}Z|_{\mathfrak{b}})\not=0.$ This shows that $\mathfrak{b}$ cannot be unimodular, and we arrive at a contradiction. It remains to show the existence of $Z$. We complete the argument in the four steps below.
	
	\noindent {\em\bf Step 1.} Recall that we denote by $W$  the Weyl group of $\mathfrak{g}.$ For $Y\in \mathfrak{a},$ $w\in W_g$ and $\alpha \in \Delta$ we have $\alpha (Y) = (w\alpha) (wY).$ Indeed, since $W$ acts on $\mathfrak{a}$ by orthogonal transformations (with respect to the Killing form $\mathcal{K}$ of $\mathfrak{g}$) we obtain
	$$\alpha (Y) = \mathcal{K}(H_{\alpha}, Y) = \mathcal{K}(wH_{\alpha}, wY)=\mathcal{K}(H_{w\alpha}, wY)=w\alpha (wY),$$
	where $H_{\alpha}\in \mathfrak{a}$ denotes the root vector of $\alpha.$
	
\noindent {\em\bf Step 2.} Let $W_{h}$ be the Weyl group of $\mathfrak{h}.$ Notice that $W_{h}$ is a subgroup of the Weyl group $W$ (since $\Delta_{h}\subset \Delta$) and for any $w\in W_{h},$ we have
	$$w(\Delta_{x}^{+})=\Delta_{x}^{+}.$$ To see that, notice first, that for any $\alpha \in \Delta_{h}$
	$$\alpha (x_{h})=0 \ \Leftrightarrow \ s_{\alpha} (x_{h})=x_{h},$$
	where $s_{\alpha} \in W_{h}$ denotes the reflection induced by $\alpha.$ Now it follows from Step 1 that
	$$0<\alpha (x_{h}) = w\alpha (wx_{h}) = w\alpha (x_{h}).$$
	Therefore $\alpha \in \Delta_{x}^{+}$ if and only if $w\alpha \in \Delta_{x}^{+}.$
	
\noindent {\em\bf Step 3.} Take $\xi:=\sum_{\alpha\in \Delta_{x}^{+}}a_{\alpha}\alpha,$ where $a_{\alpha}=dim(\mathfrak{g}_{\alpha}).$ Take $w\in W$ By Proposition \ref{prop:weyl}, $W$ is isomorphic to $N_{K}(\mathfrak{a})/Z_{K}(\mathfrak{a})$ and thus there exists $k\in K$ such that $w=Adk|_{\mathfrak{a}}.$ Since 
	$$X\in \mathfrak{g}_{\alpha} \ \Leftrightarrow \ \forall_{H\in \mathfrak{a}} \ [H,X]=\alpha (H)X \ \Leftrightarrow \ \forall_{H\in \mathfrak{a}} \ \operatorname{Ad}(k)([H,X])=\operatorname{Ad}(k)(\alpha (H)X)$$
	$$\Leftrightarrow \ \forall_{H\in \mathfrak{a}} \ [wH, \operatorname{Ad}(k)(X)]=\alpha (H)\operatorname{Ad}(k)(X) \ \Leftrightarrow (**)$$
	It follows from  Step 1 that
	$$(**) \ \Leftrightarrow \ \forall_{H\in \mathfrak{a}} \ [wH, \operatorname{Ad}(k)(X)]=w\alpha (wH)\operatorname{Ad}(k)(X) $$
	$$\Leftrightarrow \ \forall_{H\in \mathfrak{a}} \ [H, \operatorname{Ad}(k)(X)]=w\alpha (H)\operatorname{Ad}(k)(X) \ \Leftrightarrow \ \operatorname{Ad}(k)(X)\in \mathfrak{g}_{w\alpha}.$$
	Therefore
	$$\textrm{dim}(\mathfrak{g}_{\alpha})=\textrm{dim}(\operatorname{Ad}(k)(\mathfrak{g}_{\alpha}))=\textrm{dim}(\mathfrak{g}_{w\alpha}).$$
This implies that for $w\in W_{h}$ we have $w\xi=\xi,$ since $w(\Delta_{x}^{+})=\Delta_{x}^{+}.$ Therefore $\xi^{\ast}$ (that is, the vector dual to $\xi$ with respect to the Killing form of $\mathfrak{g}$) is perpendicular to $\mathfrak{a}_{h},$ because $\mathfrak{a}_{h}$ is spanned by $\{ \alpha^{\ast} \ | \ \alpha\in \Delta_{h}  \}.$
	
\noindent {\em\bf Step 4.} Note that $\mathfrak{a}=\mathfrak{t}+\mathfrak{a}_{h}$ and $\mathfrak{a}_{h} \neq \mathfrak{a}$ (this is the assumption $\text{rank}_{\mathbb{R}}G>\text{rank}_{\mathbb{R}}H$, compare the Calabi-Markus phenomenon, that is, Theorem \ref{thm:calabi-m})). It follows from Step 3 that there exists $Y\in \mathfrak{t}$ that is not perpendicular to $\xi^{\ast}.$ Moreover
	$$\operatorname{Tr}(\operatorname{ad}Y|_{\mathfrak{n}_{x}})=\xi^{\ast} (Y)$$
 	is obviously nonzero.
	
 Finally, it remains to prove that having $Y\in\mathfrak{t}$ with the property $\operatorname{Tr}(\operatorname{ad}_Y|_{\mathfrak{n}_x})\not=0$ there exists $Z=Y+X\in\mathfrak{b}$ with nilpotent $X$. Note that $AN$ is a semidirect product of $A$ and $N$, therefore, the projection $\pi: AN\rightarrow A$ onto the first factor, is a homomorphism. Note that $\pi(B)=T$ (see the proof of Lemma \ref{lemma:alg-inclusion}). 
 It follows that $d\pi:\mathfrak{a}+\mathfrak{n}\rightarrow\mathfrak{a}$ is a projection as well. Hence, $Y=d\pi(Y+X)$, where $X\in\mathfrak{n}$, as required. The proof is complete.
 \end{proof}

 \subsection{The property $B\subset AN$ and completion of proof}

 In general, for the syndetic hull, the inclusion $B\subset AN$ does not hold. However, we will show that we may assume this in our context, that is, when $B$ is a syndetic hull of a discrete solvable subgroup $\Gamma$ which acts properly and co-compactly on $G/H$. To do this, we need some preparations.
 Let $G$ be a real semisimple and connected Lie group with an Iwasawa decomposition $G=KAN$. Recall that an element $g\in G$ is called
\begin{itemize}
\item {\it hyperbolic}, if $g$ is conjugate to an element in $A$,
\item {\it unipotent}, if $g$ is conjugate to an element in $N$,
\item {\it elliptic}, if $g$ is conjugate to an element in $K$.
\end{itemize}
In what follows we will use the following facts from \cite{iw} (see Subsections 10.2-10.9 in this paper).
\begin{lemma}[\cite{iw}]\label{lemma:Jordan} Each $g\in G$ has a unique decomposition
$$g=auc$$
where $a$ is hyperbolic, $u$ is unipotent and $c$ is elliptic. Moreover:
\begin{itemize}
\item $a,u$ and $c$ commute,
\item $a,u,c\in\overline{\langle g\rangle}$, where $\overline{\langle g\rangle}$ denotes the Zariski closure of $\langle g\rangle$.
\end{itemize}
\end{lemma}

\noindent
The decomposition $g=auc$ is called {\it the real Jordan decomposition}.
In our context, a real algebraic group $T$ is a torus, if $T$ is abelian and Zariski connected, and every element of $T$ is semisimple. A torus $T$ is $\mathbb{R}$-{\it split}, if every element of $T$ is diagonalizable.
\begin{lemma}[\cite{w}]\label{lemma:torus} Let $T$ be a torus. If $T_{split}$ is the maximal $\mathbb{R}$-split subtorus of $T$, and $T_{cpt}$ is the maximal compact subtorus of $T$, then 
$$T=T_{split}\cdot T_{cpt}$$
and $T_{split}\cap T_{cpt}$ is finite.
\end{lemma}

\noindent
We will also need the following well known fact (see \cite{vgs}).
\begin{proposition}\label{prop:lattice-intersection}
Let $\Gamma$ be a (co-compact) lattice in a locally compact topological group $L$, and $L_1$ be a normal subgroup. Let $\pi: L\rightarrow L/L_1$ be the natural projection onto the quotient group. Then $\Gamma\cap L_1$ is a (co-compact)  lattice in $L_{1}$ if and only if $\pi(\Gamma)\subset L/L_1$ is a (co-compact lattice) in $L/L_1$. 
\end{proposition}
Now we complete the proof of Theorem \ref{thm:Gamma}. Thus, from this point, we assume that $\Gamma$ is a solvable discrete subgroup of $G$ with a syndetic hull $B$, whose existence is guaranteed by  Lemma \ref{lemma:syndetic}. Note again, that $L=\bar\Gamma=\bar B$  is real algebraic, and, hence it is a semidirect product
  $$L=\bar T\ltimes \bar U$$
  of a torus, and a unipotent subgroup $\bar U$.  
  By Lemma \ref{lemma:torus}, $\bar T=T_{split}\cdot T_{cpt}$. 
  Clearly, $L_1=T_{split}\ltimes U$ is normal in $L$, and $L/L_1$ is compact. 
  Also, by Theorem \ref{thm:syndetic} $B$ is closed.  It follows that $B\cap L_{1}$ is normal in $B$ and $B/B\cap L_1$ is a closed subgroup in the (Lie) group  $L/L_1$. Therefore $B/B\cap L_1$ is compact. Referring to Proposition \ref{prop:lattice-intersection} we conclude that $\Gamma\cap (B\cap L_1)$ is a lattice. Also, since $B\cap L_1=B'$ is co-compact in $B$, hence $B'$ acts properly and co-compactly on $G/H$, and, therefore, so does $\Gamma'=B'\cap\Gamma$. It follows that without loss of generality we may assume the following:

 $$L=T_{split}\ltimes \bar U.$$
  In the latter case we have $l=ua$ for each element $l\in L$, and $u,a\in L$ (by Lemma \ref{lemma:Jordan}, since $L$ is Zariski closed). Then, $T_{split}$ is contained in some maximal split torus $\hat T$ of $G$, that is, in some subgroup conjugate to $A$. Replacing $L$ by a conjugate, we may assume that $T_{split}\subset A$. In other words we know that $L\cap A$ is a maximal split torus in $L$. Using this we can prove that 
  $$L=T_{split}(L\cap N) \eqno (*)$$
   (that is, $\bar U=L\cap N$) by simply repeating the proof of Lemma 10.4 by Iozzi and Witte Morris in \cite{iw}. For the convenience of the reader we repeat their argument. Given $l\in L$ we have $l=au$, and $a$ belongs to a split torus. It is known (from the general theory of solvable algebraic groups) that all maximal split tori of $L$ are conjugate via an element of $L\cap N$, so there is some $x\in L\cap N$ such that $x^{-1}ax\in A$. Then $\langle T_{split},x^{-1}ax\rangle$ being a subgroup of $A$, is a split torus. Thus, the maximality of $T_{split}$ implies that $x^{-1}ax\in  T_{split}$. Then, for $t=x^{-1}ax$ one obtains
   
    $$l=au=xtx^{-1}u=t(t^{-1}xt)x^{-1}u\in T_{split}(L\cap N).$$
    We conclude that $L=T_{split}(L\cap N)$, as required. It follows that $B\subset L\subset AN$. Now we complete the proof applying Theorem \ref{thm:non-unimodular-b}.
  
\section{Erratum}
In the published version of this paper {\it On solvable compact Clifford-Klein forms} [published in Proc. Amer. Math. Soc. 145(2017), 1819-1832] we have overseen a gap in Section 3.5. The inclusion $B\subset AN$ cannot be assumed, because the conclusion on p. 1830 that $\Gamma\cap (B\cap L_1)$ is a lattice, actually was not proved. However, the result of the paper (Theorem 1.2) is correct, and the proof goes through with a slight modification. Here is the correction. We use the notation and the references as in the considered article. 
Note that all the proofs in it go through without the assumption that $B$ has a co-compact lattice. This assumption is only used to get a contradiction with the unimodularity condition. However, by \cite{ow2} the following equality holds: $T_cB=T_cQ$, where $Q\subset AN$. Consider the natural homomorphisms 
$$\pi_1: \bar B=T_cT_{split}\bar U\rightarrow T_{split}/(T_c\cap T_{split}),\,\text{and}\,\pi_2: AN\rightarrow A.$$
One can easily check that $d\pi_1(\mathfrak{b})=d\pi_2(\mathfrak{q})\subset\mathfrak{t}_{split}$. Thus, one can apply the argument from the article to $Q$, and get $d\pi_2(\hat Z)=Y\not\in\mathfrak{a}_h$ for some $\hat Z\in\mathfrak{q}$, but then it will also yield $d\pi_1(Z)=Y\not\in\mathfrak{a}_h$ for some $Z\in\mathfrak{b}$ (now $\pi_1$ is used instead of $\pi$ in the article). We will use $Z\in\mathfrak{b}$ to show that $\mathfrak{b}$ is not unimodular. We know that $\bar{\mathfrak{b}}=\mathfrak{t}_c+\mathfrak{t}_{split}+\bar{\mathfrak{u}}$ and that $[\bar{\mathfrak{b}},\bar{\mathfrak{b}}]\subset\mathfrak{u}$. The subspace $\mathfrak{t}_c+\mathfrak{u}$ is a subalgebra in $\bar{\mathfrak{b}}$. The corresponding connected subgroup $R\subset \bar{B}$ has a co-compact lattice by  classical Mostow's theorem. Hence, $R$ is unimodular, and $\operatorname{ad}_{\mathfrak{t}_c}$-action on $\mathfrak{u}$ is trace-free. Now we argue as in the proof of Theorem 3.10 with the following modification. Since $\mathfrak{t}_{split}\subset\mathfrak{a},\mathfrak{t}_{split}+\bar{\mathfrak{u}}\subset\mathfrak{a}+\mathfrak{n}$ and $\bar{\mathfrak{u}}\subset\mathfrak{n}$, one can write $Z=O+Y+X$, where $O\in\mathfrak{t}_c,Y\in\mathfrak{t}_{split},Y+X\in\mathfrak{a}+\mathfrak{n}$. From this one obtains
$$\operatorname{Tr}(\operatorname{ad}_Z|_{\mathfrak{b}})=\operatorname{Tr}(\operatorname{ad}_{(Y+X)}|_{\mathfrak{b}}).$$
Using the equalities $\mathfrak{n}=\mathfrak{n}_h+\mathfrak{n}_x$ and $\mathfrak{n}=\mathfrak{n}_h+\mathfrak{u}+\mathfrak{c}$ one obtains
$$\operatorname{Tr}(\operatorname{ad}_{(Y+X)}|_{\mathfrak{b}})=\operatorname{Tr}(\operatorname{ad}_{(Y+X)}|_{\mathfrak{u}})=\operatorname{Tr}(\operatorname{ad}_{(Y+X)}|_{\mathfrak{n}_x}).$$
Since $X\in\bar{\mathfrak{u}}\subset\mathfrak{n}$, the operator $\operatorname{ad}_X$ is nilpotent on $\mathfrak{n}_x$. Thus, to repeat the proof of Theorem 1.2 we need to find $Z=O+Y+X$ such that $\operatorname{Tr}(\operatorname{ad}_Y|_{\mathfrak{n}_x})=\xi^*(Y)\not=0$. We take $Z$ such that $d\pi_1(Z)=Y\not\in\mathfrak{a}_h$.

Maciej Boche\'nski
Department of Mathematics and Computer Science, University of Warmia and Mazury, 
S\l\/oneczna 54, 10-710, Olsztyn, Poland
mabo@matman.uwm.edu.pl

Aleksy Tralle
Department of Mathematics and Computer Science, University of Warmia and Mazury, 
S\l\/oneczna 54, 10-710, Olsztyn, Poland
tralle@matman.uwm.edu.pl

\end{document}